\documentclass{amsart}%

\usepackage{graphicx}
\usepackage{multicol,multirow}
\usepackage{amsmath,amssymb,amsfonts}
\usepackage{mathrsfs}
\usepackage{rotating}
\usepackage{appendix}
\usepackage[numbers]{natbib}

\usepackage{tikz}
\usepackage{subfigure}
\usepackage{hyperref}

\theoremstyle{plain}
\newtheorem{theorem}{Theorem}[section]
\newtheorem{lemma}[theorem]{Lemma}
\newtheorem{pro}[theorem]{Proposition}
\newtheorem{co}[theorem]{Corollary}
\newtheorem{Que}[theorem]{Question}
\theoremstyle{definition}
\newtheorem{definition}{Definition}[section]
\theoremstyle{remark}
\newtheorem{remark}[theorem]{Remark}
\newtheorem{example}[theorem]{Example}

\newcommand{\ga}{\big|\!\big|}
\newcommand{\red}{\textcolor{red}}
\newcommand{\blue}{\textcolor{blue}}

\numberwithin{equation}{section}

\title[A Hopf Algebra on Permutations]{A Hopf algebra on permutations with a coupling product}
\thanks{H. Li and Y. Yang are partially supported by the National Natural Science Foundation of China (No. 11701339). N. Bergeron is partially supported by NSERC Discovery.}

\author[H. Li]{Huilan Li}
\address{School of Mathematics and Statistics, Shandong Normal University, Jinan, Shandong 250358, P. R. China}
\email{lihl@sdnu.edu.cn}

\author[Y. Yang]{Yishuo Yang}
\address{School of Mathematics and Statistics, Shandong Normal University, Jinan, Shandong 250358, P. R. China}
\email{1783767117@qq.com}

\author[N. Bergeron]{Nantel Bergeron}
\address{Department of Mathematics and Statistics, York University, Toronto, ON M3J 1P3}
\email{bergeron@yorku.ca}

\subjclass{05E05, 05E99 and 16T30 }

\keywords{permutation; coupling product; draw coproduct; graded Hopf algebra; monomial basis}

\begin{document}

\begin{abstract}
We define a coupling product and a draw coproduct on permutations and show that they define a graded, connected, cocommutative, free Hopf algebra $\mathbb{KS}$. In characteristic zero,
    this implies that $\mathbb{KS}$ is isomorphic to a certain cocommutative Hopf algebra associated with the dual of the coradical filtration of the Malvenuto--Reutenauer Hopf algebra on permutations. 
    The coupling product and draw coproduct are permutation analogues of the product and coproduct on the monomial basis of symmetric functions;
    therefore, the permutation basis in our presentation can be viewed as a monomial basis for $\mathbb{KS}$.
\end{abstract}

\maketitle

%%%%%%%%%%%%%%%%%%%%%%%%%%%-1%%%%%%%%%%%%%%%%%%%%%%%
\section{Introduction}
Hopf algebras were first introduced by Hopf in the 1940s~\cite{r1}. They have many interesting properties, such as freeness, cofreeness, commutativity, and cocommutativity~\cite{r2,r3}. Over time, mathematicians have found many Hopf algebras on different objects, such as permutations~\cite{r4,r5,r18}, planar trees~\cite{GL89,r6}, simple graphs~\cite{r7}, $(0,1)$-matrices~\cite{r14}, posets~\cite{r8} and parking functions~\cite{r9,r10,r11}. In the past two decades, Hopf algebras have been applied in many areas, such as Lie algebras~\cite{r12,r13} and quantum groups~\cite{r17}.

Denote by $S_n$ the symmetric group of degree $n$, which contains all permutations of $[n]:=\{1,2,\ldots,n\}$. Let $\mathbb{K}S_n$ be the vector space with basis $S_n$ over the field $\mathbb{K}$. Define $\mathbb{KS}:=\bigoplus_{n \geq 0}\mathbb{K}S_n$, where $S_0 = \{\epsilon\}$ and $\epsilon$ is the empty permutation. Then $\mathbb{KS}$ is graded, and its $n$-th component is $\mathbb{K}S_n$. In 1995, Malvenuto and Reutenauer~\cite{r18} constructed the classical Hopf algebra on permutations, where the multiplication is the shuffle product and the comultiplication is the deconcatenation coproduct. 
In 2020, Zhao and Li~\cite{r20} defined a new shuffle product and deconcatenation coproduct on permutations and proved that the vector space spanned by permutations with this new structure is also a graded, connected Hopf algebra. They then constructed the dual Hopf algebra on permutations and found closed formulas for antipodes. In 2021, Liu and Li~\cite{r21} proved that the super-shuffle product and the cut-box coproduct on $\mathbb{KS}$ also form a graded, connected Hopf algebra.

In 2005, Aguiar and Sottile~\cite{r5} constructed a natural graded, connected, cocommutative, and free Hopf algebra corresponding to the Malvenuto--Reutenauer Hopf algebra on permutations via a coradical filtration.
They then showed that it is isomorphic to the Hopf algebra of heap-ordered trees of Grossman--Larson~\cite{GL89}. Here, {\sl free} refers to the underlying algebra, just as {\sl cofree} refers to the underlying coalgebra.
Our main result is the construction of new structures on permutations, namely, a coupling product and a draw coproduct, and the proof of the following statement.

\medskip
\noindent{\bf Main Theorem.} {\it $\mathbb{KS}$ equipped with the coupling product and the draw coproduct is a graded, connected, cocommutative, and free Hopf algebra.}
\medskip

For $\mathrm{Char}(\mathbb{K})=0$, we obtain that $\mathbb{KS}$ is also isomorphic to the cocommutative Hopf algebra of Aguiar and Sottile~\cite{r5} and the Hopf algebra of heap-ordered trees of Grossman--Larson~\cite{GL89}.
The coupling product and draw coproduct on the permutation basis are analogues of the product and coproduct on the monomial basis of symmetric functions. Therefore, we can view our construction as a monomial basis for these Hopf algebras.

The organization of this paper is as follows. In Section 2, we provide the basic definitions of Hopf algebras and review the definitions and notation for permutations. Then we introduce the definitions of absolute ascents and atoms of permutations. 
In Section 3, we introduce the standard form of a sequence of distinct positive integers, define the draw coproduct $\bigtriangleup$ on $\mathbb{KS}$, and prove that $(\mathbb{KS},\bigtriangleup,\nu)$ is a graded coalgebra. 
In Section 4,  we define the coupling product $\ltimes$ on permutations and prove that $(\mathbb{KS},\ltimes^{\uparrow},\mu)$ is a graded algebra. 
In Section 5, we prove that the product $\ltimes$ and the coproduct $\bigtriangleup$ are compatible, so that $(\mathbb{KS}, \ltimes^{\uparrow},\mu, \bigtriangleup ,\nu)$ is a graded, connected bialgebra and hence a Hopf algebra. In Section~6, we show that $\mathbb{KS}$ is free and relate it to other well-known Hopf algebras. 
We end in Section 7 by explaining the analogy between our product and coproduct and the structure of the monomial basis of symmetric functions, and conclude with some open questions.

\noindent{\sc{Acknowledgement}}: We thank the referee for the suggestions that improved our presentation.

%%%%%%%%%%%%%%%%%%%-2-%%%%%%%%%%%%%%%%%%%%%%%%%%%%
\section[Preliminaries]{Preliminaries}
\subsection{Basic Definitions}
We recall some basic definitions of Hopf algebras. Let $R$ be a commutative ring and $A$ be an $R$-module.

A \emph{product} $m\colon A\otimes_R A\longrightarrow A$ and a \emph{unit} $\mu\colon R\longrightarrow A$ satisfying the commutative diagrams in Figure~\ref{fig1} make $(A,m,\mu)$ an $R$-\emph{algebra}. The algebra $A$ is \emph{graded} if there is a direct sum decomposition $A=\bigoplus_{i\geq 0}A_i$, the product satisfies $m(A_p\otimes A_q)\subseteq A_{p+q}$, and $\mu(R)\subseteq A_0$.

\begin{figure}
\begin{minipage}[t]{0.45\textwidth}
\begin{center}
\usetikzlibrary{shapes,arrows}
\tikzstyle{arrow} = [thick,->,>=stealth]
\begin{tikzpicture}[node distance=2cm]
\node (A) at(0,2) [] {$A\otimes A\otimes A$};
\node (B) at(3,2) [] {$A\otimes A$};
\node (C) at(0,0) [] {$A\otimes A$};
\node (D) at(3,0) [] {$A$};
\draw [arrow] (A) -- (B)node[midway, above] {$m\otimes {\rm{id}}$};
\draw [arrow] (C) -- (D)node[midway, above] {$m$};
\draw [arrow] (A) -- (C)node[midway, left] {${\rm{id}}\otimes m$};
\draw [arrow] (B) -- (D)node[midway, right] {$m$};
\end{tikzpicture}
\end{center}
\end{minipage}
\hfill
\begin{minipage}[t]{0.45\textwidth}
\begin{center}
\usetikzlibrary{shapes,arrows}
\tikzstyle{arrow} = [thick,->,>=stealth]
\begin{tikzpicture}[node distance=2cm]
\node (A) at(2,2) [] {$A\otimes A$};
\node (B) at(0,0) [] {$A\otimes R$};
\node (C) at(2,0) [] {$A$};
\node (D) at(4,0) [] {$R\otimes A$};
\draw [arrow] (D) -- (C)node[midway, above] {$\cong$};
\draw [arrow] (B) -- (C)node[midway, above] {$\cong$};
\draw [arrow] (A) -- (C)node[midway, right] {$m$};
\draw [arrow] (B) -- (A)node[midway, left] {${\rm{id}}\otimes \mu$};
\draw [arrow] (D) -- (A)node[midway, right] {$\mu\otimes {\rm{id}}$};
\end{tikzpicture}
\end{center}
\end{minipage}
\caption{Associative law and unit\label{fig1}}
\end{figure}

A \emph{coproduct} $\bigtriangleup\colon A\longrightarrow A\otimes_R A$ and a \emph{counit} $\nu\colon A\longrightarrow R$ satisfying the commutative diagrams in Figure~\ref{fig2} make $(A,\bigtriangleup,\nu)$ an $R$-\emph{coalgebra}. The coalgebra $A$ is \emph{graded} if there is a direct sum decomposition $A=\bigoplus_{i\geq 0}A_i$ such that $\bigtriangleup(A_n)\subseteq\bigoplus(A_r\otimes A_{n-r})$ for $n\geq0$ and $\nu(A_n)=0$ for $n\geq1$.

\begin{figure}[h]
\begin{minipage}[t]{0.45\textwidth}
\begin{center}
\usetikzlibrary{shapes,arrows}
\tikzstyle{arrow} = [thick,->,>=stealth]
\begin{tikzpicture}[node distance=2cm]
\node (A) at(0,2) [] {$A$};
\node (B) at(3,2) [] {$A\otimes A$};
\node (C) at(0,0) [] {$A\otimes A$};
\node (D) at(3,0) [] {$A\otimes A\otimes A$};
\draw [arrow] (A) -- (B)node[midway, above] {$\bigtriangleup$};
\draw [arrow] (C) -- (D)node[midway, above] {$\bigtriangleup\otimes {\rm{id}}$};
\draw [arrow] (A) -- (C)node[midway, left] {$\bigtriangleup$};
\draw [arrow] (B) -- (D)node[midway, right] {${\rm{id}}\otimes\bigtriangleup$};
\end{tikzpicture}
\end{center}
\end{minipage}
\hfill
\begin{minipage}[t]{0.45\textwidth}
\begin{center}
\usetikzlibrary{shapes,arrows}
\tikzstyle{arrow} = [thick,->,>=stealth]
\begin{tikzpicture}[node distance=1cm]
\node (A) at(2,2) [] {$A$};
\node (B) at(0,0) [] {$R\otimes A$};
\node (C) at(2,0) [] {$A\otimes A$};
\node (D) at(4,0) [] {$A\otimes R$};
\draw [arrow] (C) -- (D)node[midway, above] {$\text{id}\otimes\nu$};
\draw [arrow] (C) -- (B)node[midway, above] {$\nu\otimes \text{id}$};
\draw [arrow] (A) -- (C)node[midway, right] {$\bigtriangleup$};
\draw [arrow] (A) -- (B)node[midway, left] {$\cong$};
\draw [arrow] (A) -- (D)node[midway, right] {$\cong$};
\end{tikzpicture}
\end{center}
\end{minipage}
\caption{Coassociative law and counit\label{fig2}}
\end{figure}

Suppose $A$ is both an $R$-algebra and an $R$-coalgebra. If $\bigtriangleup$ and $\nu$ are algebra homomorphisms (equivalently, $m$ and $\mu$ are coalgebra homomorphisms), then we say the algebra and coalgebra structures on $A$ are \emph{compatible} and $(A, m, \mu, \bigtriangleup, \nu)$ is an $R$-\emph{bialgebra}. If $A=\bigoplus_{i\geq 0}A_i$ and the operations $m,\bigtriangleup,\mu,\nu$ preserve the grading as above, then $A$ is a \emph{graded $R$-bialgebra}. Finally, if $A_0\cong R$ (we say $A$ is \emph{connected}), then by Takeuchi's formula~\cite{Takeuchi} there exists a unique antipode $S\colon A\to A$, and $(A,m,\mu,\bigtriangleup,\nu,S)$ is a \emph{graded Hopf algebra}.

\subsection{Basic Notation}
With the notation $[n]:= \{1, 2, \ldots, n\}$, the set $[n]$ is naturally ordered with $1<2<\cdots<n$ and plays a role throughout this paper.
The symmetric group $S_n$ is the set of permutations, that is, the bijections $a\colon [n]\to [n]$. We represent $a$ by listing its values as $a = a_1a_2\cdots a_n$, where $a_i = a(i) \in [n]$. In particular, $S_0 = \{\epsilon\}$, where $\epsilon$ is the empty permutation. We denote $\mathbb{S}:=\biguplus_{n\geq 0}S_n$ and $\mathbb{KS}:=\bigoplus_{n \geq 0}\mathbb{K}S_n$, where $\mathbb{K}S_n$ is the vector space with basis $S_n$ over the field $\mathbb{K}$.

Let $a = a_1a_2\cdots a_n$ be any sequence of $n$ distinct positive integers. A permutation $a\in S_n$ is a special case of such a sequence.
For $1\leq i\leq n-1$, we say that $i$ is an \emph{absolute ascent} of $a$ if $a_i<a_j$ for all $i<j\leq n$. 
Note that absolute ascents are exactly the positions $i$ such that $a_i$ is smaller than all entries to its right.
Let ${\mathcal A}(a):=\{i_1,i_2,\ldots,i_{r-1}\}$ be the set of all absolute ascents of $a$, where $i_1<i_2<\cdots <i_{r-1}$. Denote
\begin{center}
  $\alpha_1=a_1a_2\cdots a_{i_1},\qquad \alpha_2=a_{i_1+1}a_{i_1+2}\cdots a_{i_2},\qquad\ldots\qquad\alpha_r=a_{i_{r-1}+1}\cdots a_n,$
\end{center}
and call $\alpha_i$ an \emph{atom} of $a$, for $1 \leq i \leq r$. 
Thus, the atoms form a canonical decomposition of $a$ into contiguous blocks determined by its absolute ascents.
To emphasize the atoms, we place a vertical bar `$|$' between the atoms of $a$ and write $a=\alpha_1|\alpha_2|\cdots|\alpha_r$; this is called the \emph{factorization} of $a$. If $a$ has no absolute ascent, we call it \emph{irreducible}.

A nonempty sequence of distinct entries is irreducible if and only if its last entry is its smallest entry. A decomposition of a sequence into irreducible sequences is the atom factorization if and only if the last entries of the irreducibles increase from left to right.
Writing $\operatorname{rev}$ for reversal, the factorization $\operatorname{rev}(a)=\operatorname{rev}(\alpha_r)\cdots\operatorname{rev}(\alpha_1)$ is the Chen--Fox--Lyndon factorization of $\operatorname{rev}(a)$; see~\cite[Theorem~5.1.5]{Lothaire97}.

\begin{example}
For the permutation $a=3214657$, the absolute ascents are $\{3,4,6\}$ and
$a=321|4|65|7$. The permutation $231$ is irreducible.
\end{example}

%%%%%%%%%%%%%%%-3-%%%%%%%%%%%%%
\section{Draw coproduct on permutations}

The total order on the positive integers ${\mathbb Z}_{>0}$ induces a total order on any finite subset $A\subset {\mathbb Z}_{>0}$. Let $n=|A|$. We define ${\rm st}_A\colon A\to [n]$ to be the {\bf unique} order-preserving bijection. We can extend the definition of ${\rm st}_A$ to any sequence $b=b_1b_2\ldots b_\ell$ of distinct positive integers such that $\{b_1,b_2,\ldots,b_\ell\}\subseteq A$, 
by setting ${\rm st}_A(b)={\rm st}_A(b_1){\rm st}_A(b_2)\cdots {\rm st}_A(b_\ell)$. 
In particular, for the empty sequence, ${{\rm st}_A}(\epsilon)=\epsilon$.
We define ${\rm Set}(b)=\{b_1,b_2,\ldots,b_\ell\}$ and note that ${{\rm st}}_{{\rm Set}(b)}(b)$
is always a permutation in $S_\ell$. We define ${\rm st}(b) =  {{\rm st}}_{{\rm Set}(b)}(b)$, which is always a permutation.

\begin{lemma} \label{lem:stA}
Let $a=a_1a_2\cdots a_n$ be a sequence of distinct positive integers, and let $A={\rm Set}(a)$.
If $a=\alpha_1|\alpha_2|\cdots |\alpha_r$, then ${\rm st}_A(a)={\rm st}_A(\alpha_1)|{\rm st}_A(\alpha_2)|\cdots |{\rm st}_A(\alpha_r)$.
\end{lemma}

\begin{proof}
Since the map ${{\rm st}_A}$ preserves the relative order of the entries of $a$, it preserves the positions of absolute ascents.
\end{proof}

Note that in Lemma \ref{lem:stA}, ${\rm st}_A(a)={\rm st}(a)$ is a permutation, but in general, the atoms ${\rm st}_A(\alpha_i)\ne {\rm st}(\alpha_i)$.
It is thus important to be careful with the use of the notation ${\rm st}$ without an index set.

\begin{example}
For $a=423879$, we have $A=\{2,3,4,7,8,9\}$ and 
${{\rm st}_A}(423879)=312546$ is a permutation of degree 6. Moreover, $a=42|3|87|9$ and ${\rm st}_A(a)=31|2|54|6$. Here, ${\rm st}_A(42)=31\ne 21={\rm st}(42)$.
\end{example}

For $a=\alpha_1|\alpha_2|\cdots |\alpha_r$ and $K=\{i_1,i_2,\ldots,i_\ell\}\subseteq [r]$ with $i_1<i_2<\cdots<i_\ell$, we define $\alpha_K=\alpha_{i_1}|\alpha_{i_2}|\cdots |\alpha_{i_\ell}$, which is a factorization into atoms.
\begin{example}
For $a=\alpha_1|\alpha_2|\alpha_3 |\alpha_4=42|3|87|9$ and $K=\{1,4\}\subseteq [4]$, we have $\alpha_{K}=\alpha_1|\alpha_4=42|9$. 
\end{example}
\begin{pro}\label{pro:st}
 Let $a=\alpha_1|\alpha_2|\cdots |\alpha_r$ be a permutation. If $L\subseteq K\subseteq [r]$, then
  $${{\rm st}}(\alpha_L)={{\rm st}}({{\rm st}}({\alpha_K)}_{{\rm st}_K(L)}).$$
\end{pro}

\begin{proof}
Let $K=\{k_1,k_2,\ldots,k_\ell\}$ and $L=\{k_{i_1},k_{i_2},\ldots,k_{i_s}\}$. We assume that $k_1<k_2<\cdots<k_\ell$ and $i_1<i_2<\cdots< i_s$.
Note that ${{\rm st}_K(L)}=\{i_1,i_2,\ldots, i_s\}$.
Let $A={\rm Set}(\alpha_K)$, so that, using Lemma~\ref{lem:stA} twice, we have
\begin{align*}
  {{\rm st}}( {{\rm st}}(\alpha_K)_{{\rm st}_K(L)} )=&\  {{\rm st}}\big({{\rm st}}_A(\alpha_K)_{\{i_1,i_2,\ldots, i_s\}})\\
   	 =&\  {{\rm st}}\Big(\big({{\rm st}}_A(\alpha_{k_1})\big|{{\rm st}}_A(\alpha_{k_2})\big|\cdots\big| {{\rm st}}_A(\alpha_{k_\ell})\big)_{\{i_1,i_2,\ldots, i_s\}}\Big)\\
   	 =&\  {{\rm st}}\big({{\rm st}}_A(\alpha_{k_{i_1}})|{{\rm st}}_A(\alpha_{k_{i_2}})|\cdots| {{\rm st}}_A(\alpha_{k_{i_s}})\big)\\
   	 =&\  {{\rm st}}({{\rm st}}_A(\alpha_{k_{i_1}}|\alpha_{k_{i_2}}|\cdots| \alpha_{k_{i_s}})) =  {{\rm st}}({{\rm st}}_A(\alpha_L)) =  {{\rm st}}(\alpha_L).
\end{align*}
\vskip-20pt
\end{proof}

\begin{example}
  Let $a=423879=42|3|87|9=\alpha_1|\alpha_2|\alpha_3|\alpha_4$ and $L=\{1,3\}\subseteq K=\{1,3,4\}\subseteq [4]$. Then
  ${{\rm st}}(\alpha_K)={{\rm st}}(42|87|9)=21|43|5, \ {\rm st}_K(L)=\{1,2\}$ and
   $${{\rm st}}({{\rm st}}(\alpha_K)_{{\rm st}_{K}(L)})={{\rm st}}((21|43|5)_{\{1,2\}})= {{\rm st}}(21|43) =2143= {{\rm st}}(42|87)=   {{\rm st}}(\alpha_L).$$
\end{example}

\begin{definition} \label{def:copro}
  \rm{The {\emph{draw coproduct}}  $\bigtriangleup$ on $\mathbb{KS}$ is defined on the basis of permutations by}
\begin{equation*}
    \bigtriangleup(a)=\sum_{K \subseteq[r]} {{\rm st}}(\alpha_{K}) \otimes {{\rm st}}(\alpha_{[r] \backslash K}),
\end{equation*}
where $a=\alpha_1|\alpha_2|\cdots |\alpha_r.$   The \emph{counit} $\nu\colon\mathbb{KS}\rightarrow \mathbb{K}$ is defined by
$$
  \nu(a)=\left\{
\begin{aligned}
  1 &,\quad a=\epsilon, \\
  0 &,\quad \text{otherwise}.
\end{aligned}
\right.
$$
\end{definition}

It is immediate that the coproduct $\bigtriangleup$ is cocommutative and $\bigtriangleup(\epsilon)=\epsilon\otimes\epsilon$.

\begin{example}
  \rm{For the permutation $a=31|2|4$, we have}
\begin{align*}
\bigtriangleup(31|2|4)=&\ \epsilon\otimes 31|2|4+21\otimes1|2+1\otimes21|3+1\otimes31|2+31|2\otimes1+21|3\otimes1\\
   &+1|2\otimes21+31|2|4\otimes \epsilon.
\end{align*}

\end{example}

\begin{theorem}
  $(\mathbb{KS},\bigtriangleup,\nu)$ is a graded coalgebra.
\end{theorem}

\begin{proof}
Suppose $a=\alpha_1|\alpha_2|\cdots |\alpha_r.$ We compute
 \begin{align*}
   (\bigtriangleup \otimes {\rm{id}} )\circ \bigtriangleup(a) & =(\bigtriangleup \otimes {\rm{id}} )\Bigg(\sum_{K \subseteq[r]} {{\rm st}}(\alpha_{K}) \otimes {{\rm st}}(\alpha_{[r] \backslash K})\Bigg )\\
   &=\sum_{K \subseteq [r]}\bigtriangleup({{\rm st}}(\alpha_K)) \otimes {{\rm st}}(\alpha_{[r]\setminus K}).
 \end{align*}
 For $K = \{k_1 < k_2 < \cdots < k_m\}$, the definition of $\bigtriangleup$ gives
 \[
   \bigtriangleup({{\rm st}}(\alpha_K))= \sum_{L \subseteq [m]} {{\rm st}}({{\rm st}}(\alpha_K)_L) \otimes {{\rm st}}({{\rm st}}(\alpha_K)_{[m]\setminus L}).
 \]
 We reindex this sum by $J\subseteq K$, setting $L={{\rm st}}_K(J)$. Substituting this into the preceding calculation and applying Proposition~\ref{pro:st}, we obtain
 \begin{align*}
   (\bigtriangleup \otimes {\rm{id}} )\circ \bigtriangleup(a)
   &=\sum_{J\subseteq K \subseteq[r]}  {{\rm st}}( {{\rm st}}(\alpha_K)_{{\rm st}_K(J)})\otimes {{\rm st}}({{\rm st}}(\alpha_K)_{{\rm st}_K(K\backslash J)})\otimes{{\rm st}}(\alpha_{[r] \backslash K})\\
   &= \sum_{J\subseteq K \subseteq[r]} {{\rm st}}(\alpha_{J}) \otimes{{\rm st}}(\alpha_{K \backslash J})\otimes{{\rm st}}(\alpha_{[r] \backslash K}),\\
   \intertext{and since $J$, $K\backslash J$ and $[r]\backslash K$ run over all disjoint decompositions of $[r]$, we can rewrite this as}
   &= \sum_{J\uplus L\uplus M=[r]}{{\rm st}}(\alpha_{J})\otimes{{\rm st}}(\alpha_{L})\otimes{{\rm st}}(\alpha_{M}).
 \end{align*}
 A similar computation gives
\begin{align*}
 ({\rm{id}} \otimes \bigtriangleup)\circ \bigtriangleup(a) =&\ ({\rm{id}}\otimes \bigtriangleup) \Bigg(\sum_{J \subseteq[r]} {{\rm st}}(\alpha_{J}) \otimes {{\rm st}}(\alpha_{[r] \backslash J})\Bigg )\\
 =& \sum_{J\uplus L\uplus M=[r]}{{\rm st}}(\alpha_{J})\otimes{{\rm st}}(\alpha_{L})\otimes{{\rm st}}(\alpha_{M}). 
 \end{align*}
 
 Thus,
$(\bigtriangleup \otimes {\rm{id}} )\circ \bigtriangleup= ({\rm{id}} \otimes \bigtriangleup)\circ \bigtriangleup.$ 
It is clear that $\bigtriangleup$ is graded, and it is easy to verify that $\nu$ is a counit. 
\end{proof}

%%%%%%%%%%%%%%%%%%%%%-4-%%%%%%%%%%%%%%%%%%%%%
\section[Coupling product on permutations]{Coupling product on permutations}
The \emph{shift-up by $n$} map ${}^{\uparrow^n}\colon[m]\to\{1+n,2+n,\ldots,m+n\}$ that sends $i$ to $i+n$ is order-preserving. We extend this map to any sequence of positive integers
$\beta=b_1b_2\cdots b_\ell$ by setting $\beta^{\uparrow^n}=(b_1+n)(b_{2}+n)\cdots (b_\ell+n)$. Since the shift-up by $n$ is order-preserving, it preserves the positions of absolute ascents in any sequence of positive integers.
In particular,
\begin{lemma}
If $b=\beta_1|\beta_2|\cdots|\beta_s$, then $b^{\uparrow^n}=\beta_1^{\uparrow^n}|\beta_2^{\uparrow^n}|\cdots|\beta_s^{\uparrow^n}$.
\end{lemma}
\begin{definition}\label{lem:coupling}
We define the \emph{coupling product} $\ltimes^{\uparrow}$ to be the bilinear operation on $\mathbb{KS}$ that sends any two  permutations $a=\alpha_1|\alpha_2|\cdots|\alpha_r$ and $b=\beta_1|\beta_2|\cdots|\beta_s$ of respective degrees $n$ and $m$ to
$$
a\ltimes^\uparrow b:=a\ltimes b^{\uparrow^n}
:=\sum_{\substack{L \subseteq[r] \\ f\colon L \hookrightarrow[s]}}\ddot{\alpha}_1|\ddot{\alpha}_2|\cdots|\ddot{\alpha}_r|\beta_{t_1}^{\uparrow^n}|\beta_{t_2}^{\uparrow^n}
|\cdots|\beta_{t_z}^{\uparrow^n},
$$
where the sum ranges over all choices of a subset $L \subseteq[r] $ and an injection $f\colon L \hookrightarrow[s]$, and where $\{t_1<t_2<\cdots<t_z\}=[s]\setminus f(L)$. The atoms $\ddot{\alpha}_i$ are defined by
$$
\ddot{\alpha}_i=
\begin{cases}
\beta_{f(i)}^{\uparrow^n}\alpha_i, & i\in L, \\
\alpha_i, & i\notin L,
\end{cases}
\hspace{1cm} i=1,2,\ldots,r.
$$
\end{definition}
\begin{remark}\label{remark4.2}
There is also a more symmetric way of understanding $a\ltimes b^{\uparrow^n}$ as a sum over all
bipartite matchings between the atoms of $a$ and the atoms of $b^{\uparrow^n}$. Matched atoms
are merged by appending the $a$-atom to the right of the $b^{\uparrow^n}$-atom. Then all atoms -- unmerged and merged -- are concatenated in the order of increasing last entries. That is, each term in $a\ltimes b^{\uparrow^n}$ corresponds to a unique subset of the set 
$$
\{\underbrace{\alpha_i,\beta_j^{\uparrow^n}}_{\text{unmerged}},\underbrace{\beta_j^{\uparrow^n}\alpha_i}_{\text{merged}}\mid i\in [r], j\in [s]\}
$$
that contains unmerged and merged atoms and involves every atom of $a$ and $b^{\uparrow^n}$ exactly once; conversely, any such subset corresponds to a term.
\end{remark}

The \emph{unit} $\mu:\mathbb{K}\rightarrow \mathbb{KS}$ for the coupling product $\ltimes^{\uparrow}$ is given by $\mu(1)=\epsilon$. The product $\ltimes^{\uparrow}$ is noncommutative. We will see that this operation is associative in Theorem~\ref{thm:ass}.

\begin{example}
Let $a=\red{1}|\red{2}$ and $b=\blue{31}|\blue{2}|\blue{4}$. Then $[r]=\{1,2\}, [s]=\{1,2,3\}$, and the choices of $L$ and $f$, together with all corresponding terms, are as follows:
\begin{equation*}
\begin{array}{l@{\hspace{0.5cm}} r@{\hspace{0.5cm}} r l l}
L=\emptyset, &  f:\emptyset\rightarrow\emptyset, & \ddot{\alpha}_1|\ddot{\alpha}_2|\beta_{1}^{\uparrow^2}|\beta_{2}^{\uparrow^2}|\beta_{3}^{\uparrow^2}
&={\alpha}_1|{\alpha}_2|\beta_{1}^{\uparrow^2}|\beta_{2}^{\uparrow^2}|\beta_{3}^{\uparrow^2}&=\ {\textcolor{red}{1}}|{\textcolor{red}{2}}|{\textcolor{blue}{53}}|{\textcolor{blue}{4}}|{\textcolor{blue}{6}},\\

L=\{1\}, &f:1\mapsto 1, & \ddot{\alpha}_1|\ddot{\alpha}_2|\beta_{2}^{\uparrow^2}|\beta_{3}^{\uparrow^2}&=
\beta_{1}^{\uparrow^2}{\alpha}_1|{\alpha}_2|\beta_{2}^{\uparrow^2}|\beta_{3}^{\uparrow^2}&=\ 
{\textcolor{blue}{53}}{\textcolor{red}{1}}|{\textcolor{red}{2}}|{\textcolor{blue}{4}}|{\textcolor{blue}{6}},\\

& f:1\mapsto 2, &\ddot{\alpha}_1|\ddot{\alpha}_2|\beta_{1}^{\uparrow^2}|\beta_{3}^{\uparrow^2}&=
\beta_{2}^{\uparrow^2}{\alpha}_1|{\alpha}_2|\beta_{1}^{\uparrow^2}|\beta_{3}^{\uparrow^2}&=
\ {\textcolor{blue}{4}} {\textcolor{red}{1}}|{\textcolor{red}{2}}|{\textcolor{blue}{53}}|{\textcolor{blue}{6}},\\

& f:1\mapsto 3, & \ddot{\alpha}_1|\ddot{\alpha}_2|\beta_{1}^{\uparrow^2}|\beta_{2}^{\uparrow^2}&=
\beta_{3}^{\uparrow^2}{\alpha}_1|{\alpha}_2|\beta_{1}^{\uparrow^2}|\beta_{2}^{\uparrow^2}\ &=
\ {\textcolor{blue}{6}}{\textcolor{red}{1}}|{\textcolor{red}{2}}|{\textcolor{blue}{53}}|{\textcolor{blue}{4}},\\

L=\{2\}, & f:2\mapsto 1, & \ddot{\alpha}_1|\ddot{\alpha}_2|\beta_{2}^{\uparrow^2}|\beta_{3}^{\uparrow^2}&=
{\alpha}_1|\beta_{1}^{\uparrow^2}{\alpha}_2|\beta_{2}^{\uparrow^2}|\beta_{3}^{\uparrow^2}
&=\ {\textcolor{red}{1}}|{\textcolor{blue}{53}}{\textcolor{red}{2}}|{\textcolor{blue}{4}}|{\textcolor{blue}{6}},\\

& f:2\mapsto 2, & \ddot{\alpha}_1|\ddot{\alpha}_2|\beta_{1}^{\uparrow^2}|\beta_{3}^{\uparrow^2}&= {\alpha}_1|\beta_{2}^{\uparrow^2}{\alpha}_2|\beta_{1}^{\uparrow^2}|\beta_{3}^{\uparrow^2}
&=\ {\textcolor{red}{1}}|{\textcolor{blue}{4}}{\textcolor{red}{2}}|{\textcolor{blue}{53}}|{\textcolor{blue}{6}},\\

 &  f:2\mapsto 3, & \ddot{\alpha}_1|\ddot{\alpha}_2|\beta_{1}^{\uparrow^2}|\beta_{2}^{\uparrow^2}&=
 {\alpha}_1|\beta_{3}^{\uparrow^2}{\alpha}_2|\beta_{1}^{\uparrow^2}|\beta_{2}^{\uparrow^2}
 &=\ {\textcolor{red}{1}}|{\textcolor{blue}{6}}{\textcolor{red}{2}}|{\textcolor{blue}{53}}|{\textcolor{blue}{4}},\\
 
L=\{1,2\}, &f:1\mapsto 1,&\\
 &  2\mapsto 2, &\ddot{\alpha}_1|\ddot{\alpha}_2|\beta_{3}^{\uparrow^2}&=
 \beta_{1}^{\uparrow^2}{\alpha}_1|\beta_{2}^{\uparrow^2}{\alpha}_2|\beta_{3}^{\uparrow^2}
 &=\ {\textcolor{blue}{53}}{\textcolor{red}{1}}|{\textcolor{blue}{4}}{\textcolor{red}{2}}|{\textcolor{blue}{6}},\\
 
 &f:1\mapsto 1,&\\
 &  2\mapsto 3, &\ddot{\alpha}_1|\ddot{\alpha}_2|\beta_{2}^{\uparrow^2}&=
 \beta_{1}^{\uparrow^2}{\alpha}_1|\beta_{3}^{\uparrow^2}{\alpha}_2|\beta_{2}^{\uparrow^2}
 &=\ {\textcolor{blue}{53}}{\textcolor{red}{1}}|{\textcolor{blue}{6}}{\textcolor{red}{2}}|{\textcolor{blue}{4}},\\
 
 &f:1\mapsto 2,&\\
 &  2\mapsto 1, &\ddot{\alpha}_1|\ddot{\alpha}_2|\beta_{3}^{\uparrow^2}&=
\beta_{2}^{\uparrow^2}{\alpha}_1|\beta_{1}^{\uparrow^2}{\alpha}_2|\beta_{3}^{\uparrow^2}
 &=\ {\textcolor{blue}{4}} {\textcolor{red}{1}}|{\textcolor{blue}{53}}{\textcolor{red}{2}}|{\textcolor{blue}{6}},\\
 
 &f:1\mapsto 2,&\\
 &  2\mapsto 3, &\ddot{\alpha}_1|\ddot{\alpha}_2|\beta_{1}^{\uparrow^2}&=
\beta_{2}^{\uparrow^2}{\alpha}_1|\beta_{3}^{\uparrow^2}{\alpha}_2|\beta_{1}^{\uparrow^2}
 &=\ {\textcolor{blue}{4}} {\textcolor{red}{1}}|{\textcolor{blue}{6}}{\textcolor{red}{2}}|{\textcolor{blue}{53}},\\
 
 &f:1\mapsto 3,&\\
 &  2\mapsto 1, &\ddot{\alpha}_1|\ddot{\alpha}_2|\beta_{2}^{\uparrow^2}&=
 \beta_{3}^{\uparrow^2}{\alpha}_1|\beta_{1}^{\uparrow^2}{\alpha}_2|\beta_{2}^{\uparrow^2}
 &=\ {\textcolor{blue}{6}} {\textcolor{red}{1}}|{\textcolor{blue}{53}}{\textcolor{red}{2}}|{\textcolor{blue}{4}},\\
 
 &f:1\mapsto 3,&\\
 &  2\mapsto 2, &\ddot{\alpha}_1|\ddot{\alpha}_2|\beta_{1}^{\uparrow^2}&=
 \beta_{3}^{\uparrow^2}{\alpha}_1|\beta_{2}^{\uparrow^2}{\alpha}_2|\beta_{1}^{\uparrow^2}
 &=\ {\textcolor{blue}{6}} {\textcolor{red}{1}}|{\textcolor{blue}{4}}{\textcolor{red}{2}}|{\textcolor{blue}{53}}.\\
\end{array}
\end{equation*}
Thus,
     \begin{align*}
 a\ltimes b^{\uparrow^2}=&\ {\textcolor{red}{1}}|{\textcolor{red}{2}}|{\textcolor{blue}{53}}|{\textcolor{blue}{4}}|{\textcolor{blue}{6}}
 +{\textcolor{blue}{53}}{\textcolor{red}{1}}|{\textcolor{red}{2}}|{\textcolor{blue}{4}}|{\textcolor{blue}{6}}
 +{\textcolor{blue}{4}} {\textcolor{red}{1}}|{\textcolor{red}{2}}|{\textcolor{blue}{53}}|{\textcolor{blue}{6}}
 +{\textcolor{blue}{6}}{\textcolor{red}{1}}|{\textcolor{red}{2}}|{\textcolor{blue}{53}}|{\textcolor{blue}{4}}
 +{\textcolor{red}{1}}|{\textcolor{blue}{53}}{\textcolor{red}{2}}|{\textcolor{blue}{4}}|{\textcolor{blue}{6}}\\
 &+{\textcolor{red}{1}}|{\textcolor{blue}{4}}{\textcolor{red}{2}}|{\textcolor{blue}{53}}|{\textcolor{blue}{6}}
 +{\textcolor{red}{1}}|{\textcolor{blue}{6}}{\textcolor{red}{2}}|{\textcolor{blue}{53}}|{\textcolor{blue}{4}}
 +{\textcolor{blue}{53}}{\textcolor{red}{1}}|{\textcolor{blue}{4}}{\textcolor{red}{2}}|{\textcolor{blue}{6}}
 +{\textcolor{blue}{53}}{\textcolor{red}{1}}|{\textcolor{blue}{6}}{\textcolor{red}{2}}|{\textcolor{blue}{4}}
 +{\textcolor{blue}{4}} {\textcolor{red}{1}}|{\textcolor{blue}{53}}{\textcolor{red}{2}}|{\textcolor{blue}{6}}\\
 &+{\textcolor{blue}{4}} {\textcolor{red}{1}}|{\textcolor{blue}{6}}{\textcolor{red}{2}}|{\textcolor{blue}{53}}
 +{\textcolor{blue}{6}} {\textcolor{red}{1}}|{\textcolor{blue}{53}}{\textcolor{red}{2}}|{\textcolor{blue}{4}}
 +{\textcolor{blue}{6}} {\textcolor{red}{1}}|{\textcolor{blue}{4}}{\textcolor{red}{2}}|{\textcolor{blue}{53}}.
    \end{align*}
\end{example}
Let $a=a_1a_2\cdots a_n=\alpha_1|\alpha_2|\cdots |\alpha_r$ and $b=b_1b_2\cdots b_m=\beta_1|\beta_2|\cdots|\beta_s$ be two sequences of distinct positive integers. If $a_i<b_j$ for all $i\in [n]$ and $j\in [m]$, we can also use Definition \ref{lem:coupling} to define the coupling product $a\ltimes b$ of the sequences $a$ and $b$. 

\begin{theorem} \label{thm:ass}
$(\mathbb{KS},\ltimes^{\uparrow},\mu)$ is a graded algebra.
\end{theorem}

\begin{proof}
Let $a=\alpha_1|\alpha_2|\cdots|\alpha_r$, $b=\beta_1|\beta_2|\cdots|\beta_s$ and $c=\gamma_1|\gamma_2|\cdots|\gamma_t$ be permutations of degrees $n$, $m$ and $\ell$, respectively.

Using Remark \ref{remark4.2}, each term in $a\ltimes b^{\uparrow^n}$ must be given by a subset of
$$
\{\underbrace{\alpha_i,\beta_j^{\uparrow^n}}_{\text{unmerged}},\underbrace{\beta_j^{\uparrow^n}\alpha_i}_{\text{merged}}\mid 1\leq i\leq r, 1\leq j\leq s\},
$$
where each atom of $a$ and $b^{\uparrow^n}$ appears exactly once, and any such subset yields a term. Using Remark \ref{remark4.2} again, the unmerged atoms of the terms in $(a\ltimes b^{\uparrow^n})\ltimes c^{\uparrow^{n+m}}$ are $$\{\alpha_i,\beta_j^{\uparrow^{n}},\beta_j^{\uparrow^{n}}\alpha_i,\gamma_k^{\uparrow^{n+m}}\mid 1\leq i\leq r, 1\leq j\leq s, 1\leq k\leq t\},$$ and the merged atoms are $$\{\gamma_k^{\uparrow^{n+m}}\alpha_i,\gamma_k^{\uparrow^{n+m}}\beta_j^{\uparrow^{n}},\gamma_k^{\uparrow^{n+m}}(\beta_j^{\uparrow^{n}}\alpha_i)\mid 1\leq i\leq r, 1\leq j\leq s, 1\leq k\leq t\}.$$
Thus, each term in $(a\ltimes b^{\uparrow^n})\ltimes c^{\uparrow^{n+m}}$ must be given by a subset of

$$A=\left\{ \alpha_i,\beta_j^{\uparrow^{n}},\beta_j^{\uparrow^{n}}\alpha_i,\gamma_k^{\uparrow^{n+m}},
\gamma_k^{\uparrow^{n+m}}\alpha_i,\gamma_k^{\uparrow^{n+m}}\beta_j^{\uparrow^{n}},\gamma_k^{\uparrow^{n+m}}(\beta_j^{\uparrow^{n}}\alpha_i)
\begin{array}{|c} 1\leq i\leq r,\\1\leq j\leq s,\\1\leq k\leq t \end{array} \right\},
$$
 and each atom in $\alpha$, $\beta^{\uparrow^{n}}$ and $\gamma^{\uparrow^{n+m}}$ appears exactly once.

Obviously, $a\ltimes(b\ltimes c^{\uparrow^{m}})^{\uparrow^{n}}=a\ltimes(b^{\uparrow^{n}}\ltimes c^{\uparrow^{n+m}})$, where the left-hand side is the coupling product of permutations and the right-hand side is the coupling product of sequences. 
Using Remark \ref{remark4.2}, each term in $b^{\uparrow^{n}}\ltimes c^{\uparrow^{n+m}}$ must be given by a subset of
$$
\{\underbrace{\beta_j^{\uparrow^{n}},\gamma_k^{\uparrow^{n+m}}}_{\text{unmerged}},\underbrace{\gamma_k^{\uparrow^{n+m}}\beta_j^{\uparrow^{n}}}_{\text{merged}}\mid 1\leq j\leq s, 1\leq k\leq t\},
$$
where each atom of $b^{\uparrow^{n}}$ and $c^{\uparrow^{n+m}}$ appears exactly once.
Using Remark \ref{remark4.2} again, the unmerged atoms of the terms  in $a\ltimes(b^{\uparrow^{n}}\ltimes c^{\uparrow^{n+m}})$  are $$\{\alpha_i,\beta_j^{\uparrow^{n}},\gamma_k^{\uparrow^{n+m}},\gamma_k^{\uparrow^{n+m}}\beta_j^{\uparrow^{n}}\mid 1\leq i\leq r, 1\leq j\leq s, 1\leq k\leq t\}$$ and the merged atoms are  $$\{\beta_j^{\uparrow^{n}}\alpha_i,\gamma_k^{\uparrow^{n+m}}\alpha_i,(\gamma_k^{\uparrow^{n+m}}\beta_j^{\uparrow^{n}})\alpha_i\mid 1\leq i\leq r, 1\leq j\leq s, 1\leq k\leq t\}.$$ 
Thus, each term in $a\ltimes(b^{\uparrow^{n}}\ltimes c^{\uparrow^{n+m}})$
 must be given by a subset of the set 
$$B =\left\{\alpha_i,\beta_j^{\uparrow^{n}},\gamma_k^{\uparrow^{n+m}},\gamma_k^{\uparrow^{n+m}}\beta_j^{\uparrow^{n}},\beta_j^{\uparrow^{n}}\alpha_i,\gamma_k^{\uparrow^{n+m}}\alpha_i,(\gamma_k^{\uparrow^{n+m}}\beta_j^{\uparrow^{n}})\alpha_i
\begin{array}{|c} 1\leq i\leq r,\\1\leq j\leq s,\\1\leq k\leq t \end{array}\right\},$$
 and each atom in $\alpha$, $\beta^{\uparrow^{n}}$ and $\gamma^{\uparrow^{n+m}}$ appears exactly once.

Since concatenation is associative, i.e., $\gamma_k^{\uparrow^{n+m}}(\beta_j^{\uparrow^{n}}\alpha_i)=(\gamma_k^{\uparrow^{n+m}}\beta_j^{\uparrow^{n}})\alpha_i$, we have $A=B$. This shows that the terms of $(a\ltimes b^{\uparrow^n})\ltimes c^{\uparrow^{n+m}}$ are the same as those of $a\ltimes(b\ltimes c^{\uparrow^{m}})^{\uparrow^{n}}$, and we obtain
$$
(a\ltimes b^{\uparrow^n})\ltimes c^{\uparrow^{n+m}} = a\ltimes(b\ltimes c^{\uparrow^{m}})^{\uparrow^{n}}.
$$
Hence, the coupling product is associative, and $(\mathbb{KS},\ltimes^{\uparrow},\mu)$ is a graded algebra.
\end{proof}

%%%%%%%%%%%%%%%%%-5-%%%%%%%%%%%%%%%
\section{Hopf algebra on permutations}
In this section, we prove that the graded algebra $(\mathbb{KS},\ltimes^{\uparrow},\mu)$ and the graded coalgebra $(\mathbb{KS},\bigtriangleup,\nu)$ are compatible, so that $(\mathbb{KS}, \ltimes^{\uparrow},\mu, \bigtriangleup ,\nu)$ is a connected graded Hopf algebra. The proof of compatibility relies on the following proposition.

\begin{pro}\label{prop:stcoupling}
Let $a=\alpha_1|\alpha_2|\cdots|\alpha_r$ and $b=\beta_1|\beta_2|\cdots|\beta_s$ be two permutations of degrees $n$ and $m$, respectively. 
For any $J\subseteq [r]$ and $K\subseteq [s]$, let $d$ be the degree of the permutation ${{\rm st}} (\alpha_{J})$. Then
$$ {{\rm st}} (\alpha_{J}\ltimes {\beta}_{K}^{\uparrow^n})={{\rm st}} (\alpha_{J})\ltimes {{\rm st}} (\beta_{K}^{\uparrow^n})^{\uparrow^d}={{\rm st}} (\alpha_{J})\ltimes {{\rm st}} (\beta_{K})^{\uparrow^d}.$$
\end{pro}

\begin{proof} 
Let $J=\{j_1,\ldots,j_p\}$ and $K=\{k_1,\ldots,k_q\}$ with $j_1<\cdots<j_p$ and $k_1<\cdots<k_q$. 
Using Definition~\ref{lem:coupling}, any term of $\alpha_{J}\ltimes {\beta}_{K}^{\uparrow^n}$ is of the form
$$\ddot{\alpha}_{j_1}|\ddot{\alpha}_{j_2}|\cdots|\ddot{\alpha}_{j_p}|\beta_{k_{t_1}}^{\uparrow n}|\beta_{k_{t_2}}^{\uparrow n}|\cdots|\beta_{k_{t_z}}^{\uparrow n},$$
where $\{k_{t_1}<k_{t_2}<\cdots<k_{t_z}\}=K\setminus f(L)$
for some $L\subseteq J$ and some injection $f\colon L\hookrightarrow K$. For convenience, we write $\beta_{K\setminus f(L)}^{\uparrow n}$ as a shorthand for $\beta_{k_{t_1}}^{\uparrow n}|\beta_{k_{t_2}}^{\uparrow n}|\cdots|\beta_{k_{t_z}}^{\uparrow n}$. Thus, $$\ddot{\alpha}_{j_1}|\ddot{\alpha}_{j_2}|\cdots|\ddot{\alpha}_{j_p}|\beta_{k_{t_1}}^{\uparrow n}|\beta_{k_{t_2}}^{\uparrow n}|\cdots|\beta_{k_{t_z}}^{\uparrow n}=\ddot{\alpha}_{j_1}|\ddot{\alpha}_{j_2}|\cdots|\ddot{\alpha}_{j_p}|\beta_{K\setminus f(L)}^{\uparrow n}.$$
 Let $A={\rm Set}(\alpha_J)$, $B={\rm Set}(\beta_K)$ and $C=A\cup B^{\uparrow^n}$.
Clearly, all elements of $B^{\uparrow^n}$ are strictly greater than any element of $A$. Since $d=|A|$, we have ${\rm st}_C\colon C\to[d+h]$ for $h=|B|$,
$$ {\rm st}_C\big|_A\colon A\to[d]\quad\text{and}\quad {\rm st}_C\big|_{B^{\uparrow^n}}\colon B^{\uparrow^n}\to\{d+1,\ldots,d+h\}=[h]^{\uparrow^d}.$$
So $ {\rm st}_C\big|_A= {\rm st}_A$ and ${\rm st}_C\big|_{B^{\uparrow^n}}=({\rm st}_{B^{\uparrow^n}})^{\uparrow^d}$.
Using Lemma~\ref{lem:stA}, we have that the terms of 
$ {{\rm st}} (\alpha_{J}\ltimes {\beta}_{K}^{\uparrow^n}) $ are of the form
\begin{equation}\label{eq:couplingterm}
\begin{aligned}
   {\rm st}(\ddot{\alpha}_{j_1}|\cdots|\ddot{\alpha}_{j_p}|{\beta}_{K\backslash f(L)}^{\uparrow^n})
   =&\ {\rm st}_C(\ddot{\alpha}_{j_1}|\cdots|\ddot{\alpha}_{j_p}|{\beta}_{K\backslash f(L)}^{\uparrow^n})\\
   =&\ {\rm st}_C(\ddot{\alpha}_{j_1})|\cdots|{\rm st}_C(\ddot{\alpha}_{j_p})|{\rm st}_C({\beta}_{K\backslash f(L)}^{\uparrow^n}).
\end{aligned}
\end{equation}
If $\ddot{\alpha}_{j_i}={\alpha}_{j_i}$, then we have 
  $${\rm st}_C(\ddot{\alpha}_{j_i})={\rm st}_C({\alpha}_{j_i})={\rm st}_A({\alpha}_{j_i}).$$
If $\ddot{\alpha}_{j_i}=\beta_{f(j_i)}^{\uparrow^n}{\alpha}_{j_i}$, then we have 
  $${\rm st}_C(\ddot{\alpha}_{j_i})={\rm st}_C(\beta_{f(j_i)}^{\uparrow^n}{\alpha}_{j_i})={\rm st}_{B^{\uparrow^n}}(\beta_{f(j_i)}^{\uparrow^n})^{\uparrow^d}{\rm st}_A({\alpha}_{j_i})
      ={\rm st}_{B}(\beta_{f(j_i)})^{\uparrow^d}{\rm st}_A({\alpha}_{j_i}),$$
and
  $$ {\rm st}_C({\beta}_{K\backslash f(L)}^{\uparrow^n})={\rm st}_{B^{\uparrow^n}}({\beta}_{K\backslash f(L)}^{\uparrow^n})^{\uparrow^d}={\rm st}_{B}({\beta}_{K\backslash f(L)})^{\uparrow^d}.$$
This shows that the right-hand side of Equation~(\ref{eq:couplingterm}) gives the terms of 
   $${{\rm st}_A} (\alpha_{J})\ltimes {{\rm st}_{B^{\uparrow^n}}} (\beta_{K}^{\uparrow^n})^{\uparrow^d}
   ={{\rm st}} (\alpha_{J})\ltimes {{\rm st}} (\beta_{K}^{\uparrow^n})^{\uparrow^d}
   ={{\rm st}} (\alpha_{J})\ltimes {{\rm st}} (\beta_{K})^{\uparrow^d}.$$
\vskip -20pt
\end{proof}
\begin{example}
  For $a=\textcolor{red}{1|32|4}$, $b=\textcolor{blue}{31|2|4}$, $J=\{2,3\}$ and $K=\{1,3\}$, we have
\begin{align*}
   {{\rm st}} (\alpha_{J}\ltimes {\beta}_{K}^{\uparrow^4})
   =&\, {{\rm st}}(\red{32}|\red{4}\ltimes \blue{31}|\blue{4}^{\uparrow^4}) =\, {{\rm st}}(\red{32}|\red{4}\ltimes \blue{75}|\blue{8}) \\
    =&\, {{\rm st}}(\red{32}|\red{4}|\blue{75}|\blue{8})
         +{{\rm st}}(\blue{75}\red{32}|\red{4}|\blue{8})
         +{{\rm st}}(\blue{8}\red{32}|\red{4}|\blue{75})
         +{{\rm st}}(\red{32}|\blue{75}\red{4}|\blue{8})\\
    &\quad 
         +{{\rm st}}(\red{32}|\blue{8}\red{4}|\blue{75})
         +{{\rm st}}(\blue{75}\red{32}|\blue{8}\red{4})
         +{{\rm st}}(\blue{8}\red{32}|\blue{75}\red{4})\\
    =&\, {}\red{21}|\red{3}|\blue{54}|\blue{6}
         +{}\blue{54}\red{21}|\red{3}|\blue{6}
         +{}\blue{6}\red{21}|\red{3}|\blue{54}
         +{}\red{21}|\blue{54}\red{3}|\blue{6}
         +{}\red{21}|\blue{6}\red{3}|\blue{54}\\
         &+{}\blue{54}\red{21}|\blue{6}\red{3}
         +{}\blue{6}\red{21}|\blue{54}\red{3}\\
   =&\, \red{21}|\red{3}\ltimes \blue{21}|\blue{3}^{\uparrow^3}
   = \, {{\rm st}}(\red{32}|\red{4})\ltimes{{\rm st}}(\blue{75}|\blue{8})^{\uparrow^3}
   =\,  {{\rm st}} (\alpha_{J})\ltimes {{\rm st}} (\beta_{K}^{\uparrow^4})^{\uparrow^3}.
 \end{align*}
\end{example}

\begin{theorem}
  $(\mathbb{KS}, \ltimes^{\uparrow},\mu, \bigtriangleup ,\nu)$ is a graded, connected bialgebra.
\end{theorem}

\begin{proof}
 We prove that $\bigtriangleup$ is an algebra homomorphism, i.e.
  $$\bigtriangleup(a\ltimes b^{\uparrow^n})=\bigtriangleup(a)\ltimes^\uparrow \bigtriangleup(b)$$
 for any permutations $a$ and $b$ of degrees $n$ and $m$, respectively. Here, the operation $\ltimes^\uparrow$ is extended to the tensor product in the usual way of constructing tensor products of algebras.
 
 Let $a=\alpha_1|\alpha_2|\cdots|\alpha_r$ and $b=\beta_1|\beta_2|\cdots|\beta_s$.
 By Remark~\ref{remark4.2} and Definition~\ref{def:copro}, the summands of $\bigtriangleup(a\ltimes b^{\uparrow n})$ are obtained by matching atoms of $a$ with atoms of $b^{\uparrow n}$, coloring each resulting atom blue (left tensor factor) or red (right tensor factor), and standardizing each color separately.
 Give each original atom the color of the resulting atom containing it, and let $A\subseteq[r]$ and $B\subseteq[s]$ index the original blue atoms. Matched atoms have the same color, so the matching splits into independent matchings on the blue and red atoms. Conversely, fix $A\subseteq[r]$ and $B\subseteq[s]$, and choose a matching between the blue atoms $\alpha_A$ and $\beta_B^{\uparrow n}$, and another between the complementary red atoms. Their union is a valid matching between the original atoms, since the two matchings involve disjoint sets of atoms. Color each resulting merged or unmerged atom according to its constituents. This construction is inverse to restricting the original matching to the two colors. In both constructions, the atoms of each color are concatenated in increasing order of their last entries. Thus, the two resulting sequences coincide before standardization, and the standardized tensors occur with the same multiplicities.
 Thus,
\begin{align*}
\bigtriangleup(a \ltimes b^{\uparrow^n})
	=& \sum_{\substack{A\subseteq[r] \\ B\subseteq[s]}}   {{\rm st}}(\alpha_{A}\ltimes\beta_{B}^{\uparrow^n})\otimes {{\rm st}}(\alpha_{[r] \backslash A}\ltimes\beta_{[s] \backslash B}^{\uparrow^n}) ,\\
\intertext{using Proposition~\ref{prop:stcoupling} and denoting by $d_A$ the degree of ${\rm st}(\alpha_A)$,}
	=& \sum_{\substack{A\subseteq[r] \\ B\subseteq[s]}}   \big({{\rm st}}(\alpha_{A})\ltimes {\rm st}(\beta_{B})^{\uparrow^{d_A}}\big)
    		 \otimes \big({{\rm st}}(\alpha_{[r] \backslash A})\ltimes {\rm st}(\beta_{[s] \backslash B})^{\uparrow^{d_{[r]\setminus A}}}\big) \\
	=& \Big(\sum_{A\subseteq[r]}   {{\rm st}}(\alpha_{A}) \otimes {{\rm st}}(\alpha_{[r] \backslash A})\Big)   \ltimes^\uparrow
	     \Big(\sum_{B\subseteq[s]}  {\rm st}(\beta_{B}) \otimes  {\rm st}(\beta_{[s] \backslash B})\Big) \\
	=& \bigtriangleup(a)\ltimes^\uparrow \bigtriangleup(b).
\end{align*}
It is easy to verify that $\nu$ is an algebra homomorphism. We see that all operations respect the grading and that there is a unique permutation of degree 0. Hence, $(\mathbb{KS}, \ltimes^{\uparrow},\mu, \bigtriangleup ,\nu)$ is a graded, connected bialgebra.
\end{proof}

\begin{co}
  $(\mathbb{KS}, \ltimes^{\uparrow},\mu, \bigtriangleup ,\nu)$ is a graded, connected Hopf algebra.
\end{co}

\begin{proof}
Since $(\mathbb{KS}, \ltimes^{\uparrow},\mu, \bigtriangleup ,\nu)$ is a graded, connected bialgebra, Takeuchi's formula~\cite{Takeuchi} recursively constructs a unique antipode $S\colon \mathbb{KS}\to \mathbb{KS}$, and $(\mathbb{KS},\ltimes^{\uparrow},\mu, \bigtriangleup ,\nu,S)$ is a Hopf algebra.
\end{proof}

\begin{example}
  \rm{Suppose} $a=\textcolor{red}{1}$ and $b=\blue{31}|\blue{2}$. 
\begin{align*}
  \bigtriangleup(\textcolor{red} {1} \ltimes \blue{31}&|\blue{2}^{\uparrow^1}) =\bigtriangleup(\textcolor{red} {1}\ltimes\blue{42}|\blue{3})
   =\bigtriangleup(\textcolor{red}{1}|\textcolor{blue} {42}|\textcolor{blue} {3}+\textcolor{blue}{42}\textcolor{red}{1}|\textcolor{blue} {3}+\textcolor{blue} {3}\textcolor{red}{1}|\textcolor{blue} {42})  \\
   =&\ \epsilon\otimes\textcolor{red}{1}|\textcolor{blue} {42}|\textcolor{blue} {3}+\textcolor{red}{1}\otimes\textcolor{blue} {31}|\textcolor{blue} {2}+\textcolor{blue} {21}\otimes\textcolor{red}{1}|\textcolor{blue} {2}+\textcolor{blue} {1}\otimes\textcolor{red}{1}|\textcolor{blue} {32}+\textcolor{red}{1}|\textcolor{blue} {32}\otimes\textcolor{blue} {1}\\
   &+\textcolor{red}{1}|\textcolor{blue} {2}\otimes\textcolor{blue} {21}+\textcolor{blue} {31}|\textcolor{blue} {2}\otimes\textcolor{red}{1}
   +\textcolor{red}{1}|\textcolor{blue} {42}|\textcolor{blue} {3}\otimes\epsilon
   +\epsilon\otimes\textcolor{blue}{42}\textcolor{red}{1}|\textcolor{blue} {3}
   +\textcolor{blue} {32}\textcolor{red}{1}\otimes\textcolor{blue} {1}\\
   &+\textcolor{blue} {1}\otimes\textcolor{blue} {32}\textcolor{red}{1}+\textcolor{blue}{42}\textcolor{red}{1}|\textcolor{blue} {3}\otimes\epsilon
   +\epsilon\otimes\textcolor{blue} {3}\textcolor{red}{1}|\textcolor{blue} {42}+\textcolor{blue} {2}\textcolor{red}{1}\otimes\textcolor{blue} {21}+\textcolor{blue} {21}\otimes\textcolor{blue} {2}\textcolor{red}{1}+\textcolor{blue} {3}\textcolor{red}{1}|\textcolor{blue} {42}\otimes\epsilon\\
   =&\ \epsilon\otimes (\textcolor{red} {1}\ltimes\textcolor{blue} {31}|\textcolor{blue} {2}^{\uparrow^1})
       +\textcolor{blue} {21}\otimes (\textcolor{red} {1}\ltimes\textcolor{blue} {1}^{\uparrow^1})
       +\textcolor{blue} {1}\otimes (\textcolor{red} {1}\ltimes\textcolor{blue} {21}^{\uparrow^1})
       +\textcolor{blue} {31}|\textcolor{blue} {2}\otimes\textcolor{red} {1}\\
       &+\textcolor{red} {1}\otimes\textcolor{blue} {31}|\textcolor{blue} {2}
      +(\textcolor{red}{1}\ltimes\textcolor{blue} {21}^{\uparrow^1})\otimes\textcolor{blue} {1}
       +(\textcolor{red}{1}\ltimes\textcolor{blue} {1}^{\uparrow^1})\otimes\textcolor{blue} {21}
       +(\textcolor{red}{1}\ltimes\textcolor{blue} {31}|\textcolor{blue} {2}^{\uparrow^1})\otimes\epsilon\\
    =&\ (\epsilon\otimes\textcolor{red} {1}+\textcolor{red} {1}\otimes\epsilon)\ltimes^\uparrow(\epsilon\otimes\textcolor{blue} {31}|\textcolor{blue} {2}+\textcolor{blue} {21}\otimes\textcolor{blue} {1}+\textcolor{blue} {1}\otimes\textcolor{blue} {21}+\textcolor{blue} {31}|\textcolor{blue} {2}\otimes\epsilon)   \\
    =&\ \bigtriangleup(\textcolor{red} {1})\ltimes^\uparrow\bigtriangleup(\textcolor{blue} {312}).
\end{align*}
\end{example}

%%%%%%%%-6-%%%%%
\section[Relation to the classical Hopf algebra on permutations]{Relation to the classical Hopf algebra on permutations}
%%%%%%%-6.1-%%%%%%%%
\subsection{Classical Hopf algebra on permutations}
In 1995, Malvenuto and Reutenauer~\cite{r18} constructed the classical Hopf algebra structure on permutations, with its multiplication defined by the (shifted) shuffle product. 
This graded, connected Hopf algebra is self-dual; in particular, its coproduct is not cocommutative. Therefore, it cannot be isomorphic to our $\mathbb{KS}$ since we have a coproduct that is cocommutative.

In 2005, Aguiar and Sottile~\cite{r5} showed that the dual of the graded Hopf algebra associated with the coradical filtration of the Malvenuto--Reutenauer Hopf algebra is a cocommutative Hopf algebra on permutations that is isomorphic to the Grossman--Larson~\cite{GL89} Hopf algebra of heap-ordered trees.
This Hopf algebra is free, graded, connected, and cocommutative. In characteristic zero, a result of Aliniaeifard and Thiem~\cite[Theorem 12]{AT22} shows that any two free, graded, connected, cocommutative Hopf algebras $A=\bigoplus_{d\ge 0} A_d$ and $B=\bigoplus_{d\ge 0} B_d$ are isomorphic if and only if $\dim A_d=\dim B_d$ for all $d\ge 0$. 
This can be deduced using older classical results, but~\cite{AT22} offers a convenient single reference. 

We now remark that our Hopf algebra $\mathbb{KS}$ is graded, connected, and cocommutative. In the next subsection, we show that it is free; therefore,~\cite{AT22} implies that it is isomorphic to the cocommutative Hopf algebra of~\cite{r5} associated with the Malvenuto--Reutenauer Hopf algebra, and also isomorphic to the Grossman--Larson~\cite{GL89} Hopf algebra of heap-ordered trees. Finding an explicit isomorphism is not trivial. One needs to find an explicit change of basis that sends the product of the cocommutative Hopf algebra associated with the Malvenuto--Reutenauer Hopf algebra to our coupling product. We leave this question open for future work.

%%%%%%%%-6.2-%%%%%%%%%
\subsection{The Hopf algebra $\mathbb{KS}$ is free} 
In this subsection, we show that $\mathbb{KS}$ is freely generated by indecomposable permutations. This will be achieved with a leading-term and triangularity argument. 

Given a permutation $a\in S_n$, its \emph{atom degree} is ${\rm ad}(a)=r$, where $r$ is the number of atoms in the factorization 
$a=\alpha_1|\alpha_2|\cdots|\alpha_r$. We start with a simple observation about the coupling product.

\begin{lemma}\label{lem:lead}
  Given $a=\alpha_1|\alpha_2|\cdots|\alpha_r$ and $b=\beta_1|\beta_2|\cdots|\beta_s$ of degrees $n$ and $m$, respectively, the permutation 
  $c=\alpha_1|\alpha_2|\cdots|\alpha_r|\beta_1^{\uparrow^n}|\beta_2^{\uparrow^n}|\cdots|\beta_s^{\uparrow^n}$ appears in the expansion of $a\ltimes b^{\uparrow^n}$ and has atom degree 
  ${\rm ad}(c)=r+s$. All other permutations $c'$ in the expansion of $a\ltimes b^{\uparrow^n}$ satisfy ${\rm ad}(c')<{\rm ad}(c)$. 
\end{lemma}

\begin{proof} 
From Definition~\ref{lem:coupling}, the permutation $c$ is obtained when $L=\emptyset$. For all other terms $c'$, we have $L\ne\emptyset$, and the number of atoms is strictly less than $r+s$.
\end{proof}

In the next theorem, we will need the notion of an indecomposable permutation. Given a permutation $a=a_1a_2\cdots a_n$ of degree $n$, a position $1\le h<n$ is a \emph{global ascent} of $a$ if for all $1\le i\le h$ and $h<j\le n$
we have $a_i<a_j$. We remark that a global ascent position is always an absolute ascent position, but the converse is not true in general. Hence, if ${\mathcal A}(a)=\{i_1,i_2,\ldots,i_{r-1}\}$ is the set of all absolute ascents of $a$ in increasing order,
then the set of all global ascents of $a$ is a subset ${\mathcal G}(a)=\{ i_{k_1}, i_{k_2}, \ldots i_{k_{\ell-1}}\}\subseteq {\mathcal A}(a)$. For every global ascent position, we assign the symbol $\ga$.
For $a=\alpha_1|\alpha_2|\cdots|\alpha_r$ we have
$$a=\underbrace{\alpha_1|\cdots|\alpha_{k_1}}_{\Gamma_1}\ga\underbrace{\alpha_{{k_1}+1}|\cdots|\alpha_{k_2}}_{\Gamma_2}\ga\cdots\ga \underbrace{\alpha_{k_{\ell-1}+1}|\cdots|\alpha_r}_{\Gamma_\ell}.$$
Hence $a=\Gamma_1\ga\Gamma_2\ga\cdots\ga\Gamma_\ell$ is called the \emph{decomposition} of $a$, and $\overline{\Gamma}_i:={\rm st}(\Gamma_i)$ is an \emph{indecomposable component} of $a$. If $a=\Gamma_1$ has a single component, then $a$ is an \emph{indecomposable permutation}.

\begin{lemma}\label{lem:irreducible}
    Let  $a=\Gamma_1\ga\Gamma_2\ga\cdots\ga\Gamma_\ell$ be a permutation of degree $n$. The indecomposable components  $\overline{\Gamma}_i={\rm st}(\Gamma_i)$  are indecomposable permutations of degree $n_i$ such that $n=n_1+n_2+\cdots+n_\ell$. Moreover, $\Gamma_i=\overline{\Gamma_i}^{\uparrow^{n_1+\cdots+n_{i-1}}}$ and we have
    $$a=\overline{\Gamma_1}\ga\overline{\Gamma_2}^{\uparrow^{n_1}}\ga\cdots\ga\overline{\Gamma_\ell}^{\uparrow^{n_1+\cdots+n_{\ell-1}}}.$$
\end{lemma}

\begin{proof} All statements are straightforward once we have shown $\Gamma_i=\overline{\Gamma_i}^{\uparrow^{n_1+\cdots+n_{i-1}}}$. For this, it suffices to note that the definition of global ascents implies that 
${\rm Set}(\Gamma_i)$ must consist of consecutive integers; in fact,
$${\rm Set}(\Gamma_i)=\{n_1+\cdots+n_{i-1}+1, n_1+\cdots+n_{i-1}+2, \ldots, n_1+\cdots+n_{i-1}+n_i\}.$$
We have ${\rm Set}(\overline{\Gamma}_i)=[n_i]$, and all relative orders are preserved between $\Gamma_i$ and $\overline{\Gamma}_i$; therefore, our claim follows.
\end{proof}

\begin{theorem} The graded algebra $\mathbb{KS}$ is freely $\ltimes^{\uparrow}$--generated by the indecomposable permutations (permutations consisting of a single indecomposable component).
\end{theorem}

\begin{proof} 
Given $a=\Gamma_1\ga\Gamma_2\ga\cdots\ga\Gamma_r$, let $\overline{\Gamma}_i={\rm st}(\Gamma_i)$ and let $n_i$ be the degree of $\overline{\Gamma}_i$.
Note that Lemma~\ref{lem:irreducible} gives us $\Gamma_i=\overline{\Gamma}_i^{\uparrow^{n_1+\cdots+n_{i-1}}}$.
For any total order refining the atom degree, Lemma~\ref{lem:lead} implies that for any permutation $a=\Gamma_1\ga\Gamma_2\ga\cdots\ga\Gamma_r$,
the leading term (of maximal atom degree) of 
    $$\overline{\Gamma}_1 \ltimes \overline{\Gamma}_2^{\uparrow^{n_1}} \ltimes \cdots \ltimes \overline{\Gamma}_r^{\uparrow^{n_1+\cdots+n_{r-1}}} =  
    {\Gamma}_1 \ltimes {\Gamma}_2 \ltimes \cdots \ltimes {\Gamma}_r$$
is $a$ itself. 

This shows that there is a triangularity relation between the basis of permutations and the set of coupling products of indecomposable permutations. 
This triangularity implies that the change of basis is unitriangular, hence invertible.
We thus obtain that the indecomposable permutations freely generate $\mathbb{KS}$.
\end{proof}

\begin{example}
All the degree-3 products of the indecomposables are
\begin{equation*}
\begin{array}{r  l}
1\ltimes 1^{\uparrow^1}\ltimes 1^{\uparrow^2}=&\underline{1\ga2\ga3}+1\ga32+21\ga3+31|2+321,\\
1\ltimes 21^{\uparrow^1}=&\underline{1\ga32}+321,\\
21\ltimes 1^{\uparrow^2}=&\underline{21\ga3}+321,\\
312=&\underline{31|2},\\ 
231=&\underline{231},\\
321=&\underline{321},
\end{array}
\end{equation*}
%Thus,
%\begin{equation*}
%\begin{array}{r c l}
%123=&1\ga2\ga3&=1\ltimes 1^{\uparrow^1}\ltimes 1^{\uparrow^2}-1\ltimes 21^{\uparrow^1}-21\ltimes 1^{\uparrow^2}-31|2+321,\\
%132=&1\ga32&=1\ltimes 21^{\uparrow^1}-321,\\
%213=&21\ga3&=21\ltimes 1^{\uparrow^2}-321,\\
%312=&31|2&\\ 231,\hspace{0.24cm}&&\\321.\hspace{0.22cm}&&
%\end{array}
%\end{equation*}
where we have underlined the leading permutation factorized into irreducibles. There is a triangularity relation between the basis $$\{123, 132, 213, 312, 231, 321\}$$ and the set of coupling products 
$$\{1\ltimes 1^{\uparrow^1}\ltimes 1^{\uparrow^2}, 1\ltimes 21^{\uparrow^1}, 21\ltimes 1^{\uparrow^2}, 312, 231,321\}$$
of indecomposable permutations $1, 21, 231, 312, 321$. 
\end{example}

\begin{co} For $\mathrm{Char}(\mathbb{K})=0$,
the Hopf algebra $\mathbb{KS}$ is isomorphic to the cocommutative Hopf algebra of Aguiar--Sottile~{\rm\cite{r5}} associated with the Malvenuto--Reutenauer Hopf algebra of permutations. Using~{\rm \cite{r5}}, it is also isomorphic to the 
Grossman--Larson~{\rm \cite{GL89}} Hopf algebra of heap-ordered trees, since the graded dimensions coincide.
\end{co}

%%%%%%%%%%%%%%%%%%-7-%%%%%%%%%%%%%%%%%
\section{Monomial basis}
In light of the universal theorem of combinatorial Hopf algebras (see~\cite{ABS06}), it is interesting and worthwhile to study various bases of combinatorial Hopf algebras. 
Very often, one starts with a special basis, such as the monomial basis for $\rm{Sym}$, the Hopf algebra of symmetric functions.
In that case, $\rm{Sym}$ has a basis $\{m_\lambda\}_{\lambda}$ indexed by integer partitions $\lambda=(\lambda_1,\lambda_2,\ldots,\lambda_r)$, where $\lambda_1\ge\lambda_2\ge\cdots\ge \lambda_r>0$ and the $\lambda_i$ are positive integers. 
Let $c_\lambda(k)$ be the multiplicity of $k$ in $\lambda$, let $c_\lambda=\prod_{k\geq1}c_\lambda(k)!$ and let $\widetilde m_\lambda=c_\lambda m_\lambda$, a scalar multiple of the monomial basis element $m_\lambda$.
$\rm{Sym}$ is a commutative algebra, and the multiplication and comultiplication of the $\widetilde m_\lambda$ are very similar to ours. Let $\lambda=(\lambda_1,\lambda_2,\ldots,\lambda_r)$ and $\mu=(\mu_1,\mu_2,\ldots,\mu_s)$ be two integer partitions. Then
$$\widetilde m_\lambda \widetilde m_\mu = \sum_{\substack{L\subseteq[r]\\ f\colon L\hookrightarrow [s]}} \widetilde m_{{\rm Sort}(\ddot\lambda_1,\ldots,\ddot\lambda_r,\mu_{[s]\setminus f(L)})},$$
where
$$
\ddot\lambda_i=\begin{cases}
\lambda_i+\mu_{f(i)}, &\text{ if } i\in L,\\
\lambda_i, &\text{ if } i\notin L
\end{cases}
$$
and ${\rm Sort}$ sorts the entries in decreasing order. One should compare the above with Definition~\ref{lem:coupling} to see the analogy clearly. 
Similarly, 
$$\Delta(\widetilde m_\lambda)= \sum_{K \subseteq[r]} \widetilde m_{\lambda_K} \otimes  \widetilde m_{\lambda_{[r] \backslash K}},$$
which should be compared with Definition~\ref{def:copro}.
These formulas follow from the expression of augmented monomials as sums over pairwise distinct variable indices. Shared indices in a product determine the partial matching, while assigning each index to one of two alphabets gives the coproduct; see~\cite[Chapter~I, \S6, Example~10, and \S5, Example~25]{Macdonald95}.

To our knowledge, our Hopf algebra $\mathbb{KS}$ is new, and we can now consider the basis of permutations in our presentation as an analogue of the monomial basis for $\mathbb{KS}$.

The power-sum symmetric functions $p_k=\sum_{i\geq1}x_i^k$ and $p_\lambda=p_{\lambda_1}\cdots p_{\lambda_r}$ satisfy
$$p_\lambda=\sum_{\pi\in\Pi_r}\widetilde m_{\lambda_\pi},$$
where $\Pi_r$ is the set of set partitions of $[r]$ and $\lambda_\pi$ has parts $\sum_{i\in B}\lambda_i$ for $B\in\pi$, sorted in decreasing order; see~\cite[Chapter~I, \S6, Example~10]{Macdonald95}. Replacing sums of parts by concatenations of atoms suggests the following basis for $\mathbb{KS}$, which we call the \emph{concatenation basis}.

\begin{definition}
Let $a=\alpha_1|\cdots|\alpha_r$ be a nonempty permutation. For a nonempty subset $B=\{i_1<\cdots<i_k\}\subseteq[r]$, set $\alpha_B^{\leftarrow}=\alpha_{i_k}\cdots\alpha_{i_1}$.
For $\pi=\{B_1,\ldots,B_\ell\}\in\Pi_r$, order the blocks so that $\min B_1<\cdots<\min B_\ell$ and define
$$a_\pi=\alpha_{B_1}^{\leftarrow}|\cdots|\alpha_{B_\ell}^{\leftarrow},
\qquad C_a=\sum_{\pi\in\Pi_r}a_\pi.$$
Set $C_\epsilon=\epsilon$.
\end{definition}

The merged sequences are irreducible and their last entries increase, so the bars in $a_\pi$ give its atom factorization. The discrete set partition contributes $a$, while all other terms have smaller atom degree. Thus, the $C_a$ form a basis by an integral unitriangular change of basis.
For permutations $a$ and $b$, a set partition of the combined atom set amounts to partitions of the two inputs together with a partial matching of their blocks, giving the concatenation rule
$$C_a\ltimes^\uparrow C_b=C_{a\,b^{\uparrow^n}},$$
where $n$ is the degree of $a$.
However, these elements need not be primitive, even for indecomposable indices: $C_{312}=312+231$ satisfies
$$\bigtriangleup(C_{312})=\epsilon\otimes C_{312}+C_{312}\otimes\epsilon+21\otimes1+1\otimes21.$$

\begin{Que} The following questions are open problems that would be interesting to work on in the future.
\begin{enumerate}
\item For $\mathrm{Char}(\mathbb{K})=0$, find an explicit isomorphism between $\mathbb{KS}$ and the cocommutative Hopf algebra on permutations defined in~{\rm \cite{r5}}.
\item In positive characteristic, is there an isomorphism as in (1)?
\item Can one construct an explicit analogue of the power-sum basis, consisting of ordered coupling products of homogeneous primitive free generators with simple combinatorial descriptions?
\item Study other bases of $\mathbb{KS}$. Is there an analogue of Schur functions in this Hopf algebra?
\item Study the characters of $\mathbb{KS}$ and the related combinatorial invariants as defined in~{\rm \cite{ABS06}}. Note that each character of $\mathbb{KS}$ defines
a graded Hopf morphism to $\rm{Sym}$ (since $\mathbb{KS}$ is cocommutative). It would be interesting to study such morphisms.
\item Can the coupling product and draw coproduct be placed in a general framework of matching constructions on labeled atoms with an associative, possibly noncommutative, merging operation? How does such a framework relate to the quasi-shuffle algebras discussed in~{\rm\cite[\S1.3]{Foissy23}}?
\end{enumerate}
\end{Que}

\end{document}